\documentclass[12pt]{article}

\usepackage{url}
\usepackage{bbm}
\usepackage{mathtools}
\usepackage{amssymb}
\usepackage{amsthm}
\usepackage{tikz}
\usepackage{empheq}
\usepackage{latexsym}
\usepackage{eurosym}
\usepackage{dsfont}
\usepackage{appendix}
\usepackage{color} 
\usepackage[unicode]{hyperref}
\usepackage{frcursive}
\usepackage[utf8]{inputenc}
\usepackage[T1]{fontenc}
\usepackage{geometry}
\usepackage{multirow}
\usepackage{todonotes}
\usepackage{lmodern}
\usepackage{anyfontsize}
\usepackage{pgfplots}
\usepackage{stmaryrd}
\usepackage{float}

\usepackage{graphicx}
\usepackage{caption}
\usepackage{subcaption}
\usepackage{enumerate}
\usepackage{breakcites}
\graphicspath{ {images/} }
\usepackage{comment}
\pgfplotsset{compat=newest}
\usepackage{algorithm,algcompatible}
\algnewcommand\INPUT{\item[\textbf{Input:}]}
\algnewcommand\OUTPUT{\item[\textbf{Output:}]}

\newtheorem{Theorem}{Theorem}[part]
\newtheorem{Definition}{Definition}[part]
\newtheorem{Proposition}{Proposition}[part]

\newtheorem{Lemma}{Lemma}[part]

\newtheorem{Remark}{Remark}[part]

\usepackage{algpseudocode}
\makeatletter \@addtoreset{equation}{section}

\@addtoreset{Definition}{section}

\@addtoreset{Theorem}{section}

\@addtoreset{Proposition}{section}

\@addtoreset{Property}{section}

\@addtoreset{Assumption}{section}

\@addtoreset{Corollary}{section}

\@addtoreset{Lemma}{section}

\@addtoreset{Remark}{section}

\@addtoreset{Example}{section}

\@addtoreset{Condition}{section}

\newcommand{\fproof}{\hfill $\square$ \bigskip}

 \title{Optimal Impulse Control for Cyber Risk Management}

 \author{Caroline {\sc Hillairet}\footnote{ENSAE IP Paris, Center for Research in Economics and Statistics, France. \texttt{caroline.hillairet@ensae.fr}}, \, Thibaut {\sc Mastrolia}\footnote{UC Berkeley, Department of Industrial engineering and Operations Research, USA. \texttt{mastrolia@berkeley.edu}} \, and\,  Wissal {\sc Sabbagh}\footnote{Laboratoire Manceau de Math\'ematiques, Institut du Risque et de l'Assurance, Le Mans Universit\'e, France. \texttt{wissal.sabbagh@univ-lemans.fr }.}}

\begin{document}

\maketitle

\begin{abstract} 
We explore an optimal impulse control problem wherein an electronic device owner strategically calibrates protection levels against cyber attacks. Utilizing epidemiological compartment models, we qualitatively characterize the dynamics of cyber attacks within the network. We determine the optimal protective measures against effective hacking by formulating and solving a stochastic control problem with optimal switching. We demonstrate that the value function for the cluster owner constitutes a viscosity solution to a system of coupled variational inequalities associated with a fully coupled reflected backward stochastic differential equation (BSDE). Furthermore, we devise a comprehensive algorithm alongside a verification procedure to ascertain the optimal timing for network protection across various cyber attack scenarios. Our findings are illustrated through numerical approximations employing deep Galerkin methods for partial differential equations (PDEs). We visualize the optimal protection strategies  and efficiency in the context of two distinct attack scenarios: (1) a constant cyber attack, (2) an exogenous cyber attack strategy modeled with a Poisson process.
\vspace{3mm}

\vspace{3mm}

 
\end{abstract}

\section{Introduction}

With the widespread deployment of connected systems and the ever-increasing digitalization of our economy and society, the risk of cyber failures is omnipresent for individuals, businesses, and institutions across both public and private sectors. These cyber failures are diverse and complex, including incidents such as hacking, ransomware attacks, and DDoS attacks, and they carry varied consequences like data corruption, data loss, and business disruptions that can significantly impact supply chains.

Cyber risks can lead to economic failures through massive, large-scale attacks affecting numerous victims or through more targeted cyber events that weaken networks of industrial interdependencies. This threat has been amplified by recent health and geopolitical crises. For instance, the COVID-19 pandemic created additional opportunities for cybercriminals, resulting in a surge of attacks such as ransomware, as highlighted in the 2020 activity report by the National Agency for the Security of Information Systems (ANSSI) \cite{ANSSI}. More recently, the war in Ukraine has demonstrated the power of cyberattacks as instruments of warfare.
The costs of cyber-risk are escalating rapidly. Damages caused by cybercrime, estimated  by \cite{CEA} at \$1 trillion in 2021 (equivalent to 1\% of global GDP), could soar to \$10 trillion by 2025, according to several sources, such as the World Economic Forum and   Cybersecurity Ventures.\footnote{https://cybersecurityventures.com/official-cybercrime-report-2025/}.

The evolving nature of cyber-risk, its potential to become systemic, and its behavioral aspects make it challenging to establish the most effective cyber-risk management policies. Cyber-risk is inherently a human risk, necessitating a deeper understanding of the behaviors and motivations of the various actors involved. The COVID-19 crisis highlights one of the most concerning characteristics of cyber-risk: the adaptability and opportunism of hackers, as evidenced by the surge in malicious websites and fraudulent emails exploiting the pandemic.
To enhance the resilience of the economy against this growing threat, it is crucial that users adopt and consistently implement robust cyber-protection measures.
Similar to epidemiology, the approach to cyber-risk management is twofold: first, to adopt the right measures to avoid becoming a victim of the disease, and second, to prevent spreading the risk to others. Raising awareness and implementing effective protection measures are essential, even for the smallest companies. If these smaller entities are compromised, they can negatively impact larger organizations, either through a domino effect in the supply chain (as seen during the COVID-19 pandemic with the shortage of electronic components) or through a Trojan horse effect (as exemplified by the SolarWinds attack, where a breach in a supplier's security enabled the infiltration of large companies and government departments).
Prevention policy is a hot topic issue, underscored by the enactment of Europe’s Digital Operational Resilience Act (DORA), which took effect in January 2023, with full implementation slated for January 2025 or the Cyber Trust Mark in the United States. DORA requires financial entities to demonstrate the ability to withstand, respond to, and recover from significant operational disruptions related to information and communication technologies (ICT). Notably, the regulation mandates the creation of a monitoring mechanism for the service providers these companies depend on, ensuring a more resilient and secure digital ecosystem. The U.S. Cyber Trust Mark program has been introduced by the Federal Communications Commission (FCC) in July 2023. It is designed to help consumers easily identify smart devices that meet certain cybersecurity standards, ensuring that these devices are more secure against potential cyber threats.\\

Recent literature has led to significant advances in the quantification of cyber risk, particularly in relation to insurance coverage. Notable contributions in this area include works such as \cite{eling2017data,eling2020capital}, or \cite{awiszus2023modeling} among others.
 \cite{farkas2021cyber} investigate severe and extreme cyber claims using a combination of Generalized Pareto modeling and regression tree approaches.
The accumulation and contagion characteristics of cyber events, as highlighted by \cite{zeller2023accumulation}, can be modeled using epidemiological network models adapted to the specific nature of cyber risk, see \cite{nguyen2017modelling},  \cite{hillairet2021propagation,hillairet2022cyber}. Alternatively, self-excited counting models, such as those proposed by  \cite{Multi-Variate_HawkesBessy}, or marked-point processes as in \cite{zeller2023risk} may also be considered. 
These models can be  used for the pricing of cyber-insurance contracts, as in \cite{fahrenwaldt2018pricing} or \cite{hillairet2023expansion}. In this paper, we utilize several of these models, with a particular emphasis on epidemiological models, to determine optimal switching protection policies in the context of cybersecurity. \\

The economics of cybersecurity was formally established in a theoretical framework by Gordon and Loeb in their seminal 2002 paper \cite{gordon2002economics}. They proposed a model for determining the optimal allocation of a limited budget across different information sets, which are characterized by their vulnerability and potential loss in the event of a successful cyber attack. This influential article has inspired numerous subsequent studies that have refined and expanded upon its conclusions by incorporating more specific and concrete hypotheses, as exemplified by \cite{skeoch2022expanding}. In particular, \cite{gordon2020integrating} adapted this model to determine the optimal level of security within the framework of the NIST Cybersecurity Framework. The article \cite{mazzoccoli2020robustness} examines optimal investment decisions in the context of mixed insurance and investment strategies for managing cyber risk. Additionally, \cite{kolokoltsov2016mean} analyzes the response of defense systems to cyber-attacks as a stochastic game involving a large number of interacting agents.\\

This study proposes to address the challenge faced by cluster owners in balancing the costs of protecting their computer networks against cyber-attacks. The study focuses on optimizing the decision-making process regarding whether to regularly update or purchase security software. The issue at hand involves a trade-off: inadequate protection can result in substantial financial losses due to cyber incidents, which affect both the cluster owner and its customers. Conversely, implementing active protection measures can be very costly.

This work emphasizes the need for dynamic and adaptive protection strategies due to the rapid evolution of cyber threats and the behavior of both hackers and users. To address this challenge, the paper proposes a method for defining optimal protection policies that are implemented continuously over time and involve discrete sets of strategic choices. These policies are determined through the optimization of performance-cost criteria, using advanced stochastic impulse control techniques and regime switching. This approach provides a structured framework for achieving an effective balance between the costs of cyber protection and the risks of potential cyber incidents.\\

The theory of stochastic impulse control has been developed in the 70s' and early 80s' \cite{bensoussan1975nouvelles,lepeltier1979techniques,kushner1976approximations,Tang1993} by considering verification theorem and quasi-variational-inequality by using control tools developed by Bensoussan and Lions. We refer to \cite{bensoussan2011applications} for a review of the litterature on the topic or to \cite{belak2017general} for a general formulation of the problem. It has then been extended to stochastic diffusion models and mixed controlled-switching problems in for example \cite{guo2001explicit,ly2007explicit,bayraktar2010one,bayraktar2016robust,kharroubi2016optimal,oksendal2019stochastic} and linked to BSDE theory in a non-Markovian setting in \cite{chassagneux2011note,elie2014bsde}. This kind of problems has been applied in diverse field of economics: storage systems \cite{harrison1983impulse,harrison1983instantaneous}; decision-making theory with entry and exit decisions \cite{dixit1989entry,zervos2003problem,guo2005optimal}; energy storage \cite{carmona2010valuation} and energy production \cite{bayraktar2023neural}; control of portfolio in which an investor optimally intervenes in order to rebalance his portfolio
and consume a nonnegative amount of money at random chosen times  in \cite{eastham1988optimal}; optimal investment \cite{savku2022stochastic}; price formation in limit order book \cite{gayduk2018endogenous}; operational flexibility of energy assets \cite{carmona2008pricing}; commodity market in \cite{alizadeh2008markov,chen2012regime,ludkovski2011stochastic} or more recently \cite{aid2021impulse} in which the price of 
the commodity is influenced by firms' competition.\\

This study explores a stochastic epidemiological SIRS model that switches between different dynamics based on two factors: the control exerted by cluster owners (endogenous switching control) and hacking activities (exogenous and uncertain hazards), which are modeled through various attack scenarios. This research addresses a switching problem within a stochastic epidemiological framework, specifically focusing on cyber risk management in the presence of external attacks and random environment. Utilizing It\^o's calculus and the verification method, we derive a characterization of the cluster owner's value function as a viscosity solution to a system of quasi-variational inequalities, based on dynamic programming principles. We also present a practical pseudo-algorithm and a verification theorem with explicit switching conditions. Our approach includes a detailed method and a pseudo-algorithm to facilitate switching between protection policies under different attack scenarios.\\

Finally, we develop numerical approximations to simulate the
optimal protection strategies of the cluster's owner based on the use of Deep Galerkin Method \cite{DGM}. 
The Deep Galerkin Method (DGM) is a numerical technique that leverages deep learning to solve partial differential equations (PDEs). It builds on the classic Galerkin method but uses deep neural networks to approximate the solution to the PDE, making it particularly well-suited for high-dimensional problems where traditional methods struggle due to the curse of dimensionality. This method is also known for 
its capacity to handle complex domains which makes it a powerful tool for various applications, particularly in fields such as finance, physics, and engineering. \\

The paper is organized as follows.  Section 2 presents the epidemiological SIRS dynamics  used to model  cyber-attacks contagion  through the cluster, and the impact of a protection campaign. Section 3 states the optimal impulse control problem of a cluster owner facing  exogenous cyber-attacks.  It is solved using dynamic programming principle. A detailed numerical study is provided in Section 4. 



\section{Computer cluster modeling}
Throughout the paper, we consider a filtered probability space $(\Omega,\mathcal F,\mathbb{F},\mathbb P)$ endowed with a one-dimensional Brownian motion denoted by $W$. The Brownian motion is viewed as an uncertainty to determine precisely the transmission rate of the virus inside the computers' cluster. 

\paragraph{Contagion, protection and hacking.}
The computers or electronic devices  in the cluster can be in three different states, defined below:
 
\begin{itemize}
\item The class of \textbf{Susceptible (S)}:  $S_t$  denotes the proportion  at time $t$ of non-sufficiently protected and  not-yet-infected computers, thus susceptible to be attacked. 
\item The classes of \textbf{Infected (I)}:  $I_t$ denotes the proportion  at time $t$  of infected computers which in turn can contaminate other devices. 
\item The classes of \textbf{Removed (R)}:  $R_t$ denotes the proportion  at time $t$ of computers that are recovered after infection or protected by the antivirus software, and thus can not be infected anymore.  If one consider a given cyber-attack, this protection can  be  effective forever, thus leading to a SIR model. Alternatively if one consider different types of cyber-infection, the removed state is transient, leading to a SIRS model.
\end{itemize}

 The process $(S_t,I_t,R_t)_{t\geq 0}$ denote the proportions of computers in the corresponding classes with respect to the total number of computers. At each time $t$, the system has to satisfy $S_t+I_t+R_t=1$.\vspace{0.3em}

We assume that the hacker's strategy, denoted by $(a_t)_t\geq 0$, is a binary variable taking either the value $a_t=1$ if the hacker attacks the cluster or $a_t=0$ if the hacking  is inactive. When there is an attack, the intensity of attack is fixed at $\nu>0$. The response of  the cluster owner's to protect its network  is also a binary control variable denoted by $(p_t)_{t\geq 0}$ such that either he develops a dedicated protection to this attack, that is $p_t=1$ or he remains with the benchmark level of protection, that is $p_t=0$. The intensity of defense implementation is $\kappa >0$. Therefore the strategy of the cluster owner (respectively hacker) is equivalently defined by the switching times from  activating dedicated protection to  stopping it (respectively from launching an attack to stopping it). \vspace{0.5em}

The evolution of the system is the following:

\begin{itemize}
\item Computers in the class (S) can stay in the class (S) or can pass to the class (I) with fixed rate $ a \nu$ under Hacker's action $a$, or by contagion  with all infected computers with parameter $\beta$.
\item Computers in the class (S) can pass to the class (R) under the action of the cluster owner $p$ by downloading the antivirus software with a proportion $ p \kappa$ of computers in the class (S).
\item Computers in the class (I) are replaced with rate $\gamma>0$ to pass to the class (R).
\item Computers in the class (R) can pass to the class (S)  with rate $\rho \geq 0$  as the protection measure  becomes obsolete.
\end{itemize}
Note that if one want to model a given attack on a short horizon, $\rho$ could be taken as zero.
The evolution of the system under protection, hacking and contagion is summarized in the following graph.

\begin{figure}[!ht]
    \centering
\begin{tikzpicture}
\node [rectangle, draw, rounded corners=4pt, fill=blue!20, text centered] (S) at (-2,0) {\bf Susceptible};
\node [rectangle, draw, rounded corners=4pt, fill=red!20, text centered] at (3,0) (I) {\bf Infected};
\node [rectangle, draw, rounded corners=4pt, fill=green!20, text centered] (P) at (0,-5) {\bf Removed}; 
\draw[->, draw=black, thick] (S) edge [ bend left]  node[ midway] { $\underset{(\mbox{contagion})}{\beta}$}(I);
\draw[->, draw=black, thick] (S) edge  [ bend right]  node[midway] {$\underset{(\mbox{hacking})}{a\nu}$}(I);
\draw[->, draw=black, dashed, thick] (I) edge[bend left]  node[midway, right] { $\; \gamma$}(P);
\draw[->, draw=black, thick] (P) edge  [ bend right]  node[midway,left] {$\rho$}(S);
\draw[->, draw=black, dashed, thick]  (S) edge[bend right]  node[midway, left] {$\underset {(\mbox{protection})}{p \kappa}$}(P);
\end{tikzpicture}
    \caption{Cluster evolution}
    \label{fig:SIS_model}
\end{figure}

We assume that the infection rate, rather than being constant, is subject to
random shocks which are modeled by a Brownian motion $W$ as in \cite[Equation (5)]{lesniewski2020epidemic}.
Hence, the dynamics of the SIRS system  evolves as
\begin{equation}\label{SIRsystem}
\begin{cases}
dS_t=\big(\rho R_t- S_t\left( a_t \nu+ I_t \beta +  p_t \kappa \right) \big)dt  - \sigma I_t S_t dW_t.\\
dI_t= a_t \nu S_tdt + \beta S_t  I_tdt - I_t \gamma dt + \sigma I_tS_t dW_t\\
dR_t=  p_t \kappa S_t dt  + \gamma I_t dt-\rho R_t dt \\
\end{cases},
\end{equation}
where $a$ and $p$ are switching 
processes taking values into $\{0,1\}$ and specified hereafter.

\paragraph{Hacker's strategy.} 

The strategy $a$ of the hacker  is defined by   $\tilde \alpha:= (a_0, (\tilde\tau_n)_{n\geq 0})$  where $a_0 \in \{ 0, 1\}$ is the initial state  and  $(\tilde\tau_n)_{n\geq 0}$ are  the switching-times of the attack level,  with $\tilde\tau_0:=0$. The sequence $(\tilde\tau_n)_{n\geq 1}$ is an increasing sequence of  random times such that $\tilde\tau_n\longrightarrow +\infty$ when $n$ goes to $+\infty$. 
Starting with the initial state  $a_0 \in \{ 0, 1\}$, then  the state of attack at time $t$ is
 $$a_t =  \sum_{n\geq 0}  \mathbf 1_{\tilde{\tau}_{2n+1-a_0}  \leq t< {\tilde{\tau}_{2n+2-a_0}}}$$ and the intensity at time $t$  of the attack of the hacker is equal to $\nu a_t$.
 The random times $(\tilde\tau_n)_{n\geq 0}$ are assumed exogenous random times  independent to the filtration $\mathbb F$.

\paragraph{Cluster owner's strategy.} 
The  initial state  of protection $p_0 \in \{ 0, 1\}$ being given, the strategy $p$ of the cluster owner is then characterized  by  the sequence $\alpha$ of the switching-times  of the protection level $\alpha:=(\tau_n)_{n\geq 0}$ with $\tau_0:=0$. The cluster owner observes the current state of the system $(S_t,I_t,R_t)$ and we assume that he has set up a monitoring system to identify the current state of attack $a_t$, while not being able to anticipate the strategy of the hacker. More precisely, the cluster owner observes the current attack level but does not anticipate potential random changes in the attack, since he typically has no information about the hackers’ behavior. He therefore acts as if the attack level remains constant, without anticipating future changes. When the attack level changes, the cluster owner observes the shift and reacts by adapting his strategy to the new environment induced by the attack.
 In other words, the cluster owner is subjected to random switch of the environment that he undergoes without being able to anticipate it. On each random time interval $[\tilde{\tau}_n, \tilde{\tau}_{n+1}[ $ characterized by the constant attack level  $a_{\tilde{\tau}_n} = 0$ or $1$,  the cluster owner's strategy $p_t$ depends on  this attack level $a_{\tilde{\tau}_n}$: $p_t =\mathfrak p_t(a_{\tilde{\tau}_n})$ where  $\mathfrak p_t(0)$ and $\mathfrak p_t(1)$ are $\mathbb F$-adapted processes.
In terms of switching  times, the  cluster owner's strategy consists in the sequence $(\tau_n)_{n\geq 1}$ of increasing   random times, depending on the random environment of attack,  such that $\tau_n\longrightarrow +\infty$ when $n$ goes to $+\infty$.
Starting with the initial state $p_0 \in \{ 0, 1\}$, then  the state of protection at time $t$ is
 $$p_t =  \sum_{n\geq 0}  \mathbf 1_{\tau_{2n+1-p_0}  \leq t< \tau_{2n+2-p_0}}$$ and the intensity at time $t$  of the cluster owner defense is equal to $\kappa p_t$.

\vspace{0.5em}

We denote by $\mathcal{A}^p$ 
the set of admissible switching control of the cluster owner $\alpha:=(\tau_n)_{n\geq 0}$ satisfying $\mathbb E[\underset{{n\geq 1}}{\sum} e^{-\delta \tau_n }]<+\infty $, for a given strategy $\tilde \alpha = (\tilde \tau_n)_{n\geq 0}$ of the hacker. The hacker's strategy $\tilde \alpha$ determines the random environment faced by the cluster owner, who  is unable  to anticipate it. For a protection strategy $\alpha  \in  \mathcal{A}^p$, in a given environment of attacks $\tilde \alpha$, the  dynamics of the system is given by
\begin{equation}\label{systdynamics}\begin{cases}
S_t^{\alpha,\tilde \alpha}=s_0+\displaystyle\int_0^t    \rho R^{\alpha,\tilde \alpha}_s ds  -   \displaystyle\int_0^t    S^{\alpha,\tilde \alpha}_s   I^{\alpha,\tilde \alpha}_s\  \left(  \beta  ds  + \sigma   d W_s\right)
\\
\hspace{3em} -    \underset{\tau_n\leq t}{\sum}\displaystyle\int_{\tau_n}^{\tau_{n+1}\wedge t} S^{\alpha,\tilde \alpha}_s  \kappa p_s  ds -    \underset{\tilde\tau_n\leq t}{\sum}\displaystyle\int_{\tilde \tau_n}^{\tilde \tau_{n+1}\wedge t}  S_t^{\alpha,\tilde \alpha}   \nu  a_s  ds\\
I_t^{\alpha,\tilde \alpha}= i_0 +\displaystyle \int_0^t      I^{\alpha,\tilde \alpha}_s  \left(  ( \beta S^{\alpha,\tilde \alpha}_s  - \gamma) ds  + \sigma  S^{\alpha,\tilde \alpha}_s  d W_s\right)
+    \underset{\tilde\tau_n\leq t}{\sum}\displaystyle\int_{\tilde \tau_n}^{\tilde \tau_{n+1}\wedge t}  S_t^{\alpha,\tilde \alpha}   \nu  a_s  ds \\
R_t^{\alpha,\tilde \alpha}= r_0 +  \displaystyle \int_0^t      (I^{\alpha,\tilde \alpha}_s   \gamma -\rho R^{\alpha,\tilde \alpha}_s)  ds  + \underset{\tau_n\leq t}{\sum}\displaystyle\int_{\tau_n}^{\tau_{n+1}\wedge t}  \kappa p_s S^{\alpha,\tilde \alpha}_s ds,\\

S_0=s_0,\; I_0=i_0,\; R_0=r_0,
\end{cases}
\end{equation}
where $s_0+i_0+r_0=1$ and $(s_0,i_0,r_0)\in [0,1]^3$.
\vspace{0.5em}

\section{Impulse control, switching and cluster owner's optimization}
 Throughout the paper, the environment  of  attacks $\tilde \alpha$ is exogenous  and each change times $\tilde \tau_n$  in the attack level is revealed to the cluster owner only when it occurs. We deal with the optimal 
strategy for  the cluster owner. 
More precisely, the cluster owner has to solve  a two regime switching controlled SIRS system, by choosing an admissible switching control  $\alpha=(\tau_n)_{n\geq 0} \in \mathcal A^p$  that optimizes the following criteria  with initial state $(s_0, i_0)$  and initial regime $p_0$ for the cluster owner 
\begin{equation}\label{pb:optprotection} V(s_0,i_0;p_0)=\inf_{\alpha \in \mathcal A^p} \mathbb E\left[\int_0^{+\infty}e^{-\delta t} (c_I I^{\alpha,\tilde \alpha}_t + f(S^{\alpha,\tilde \alpha}_t,p_t))dt +\sum_{n\geq 1} e^{-\delta \tau_n }   g_{p_{\tau_{n-1}}, \,p_{\tau_n}}(S^{\alpha,\tilde \alpha}_{ \tau_n },I^{\alpha,\tilde \alpha}_{ \tau_n }) \right]\end{equation}
where the cost of the vaccination is 
$f(s,p)=c_V \kappa \, s  \, p $, with $c_V$ the marginal cost of the vaccination, $c_I$ the marginal cost of the infected. For $p\in\{0,1\}$ the switching cost $g_{p,\bar p}:[0,1]^2\rightarrow\mathbb R$ is a positive continuous function. We assume that
there exists a constant $c_g>0$ such that for any $(s,i)\in [0,1]^2$,
$$\frac{1}{c_g}\leq g_{p,\bar p}(s,i)\leq c_g .$$ Here, $\delta>0$ is a positive discount factor, and we  use the convention that $e^{-\delta \tau_n}=0$ when $\tau_n=\infty$. Since $S$ and $I$ are valued in $[0,1]$, the expectation defining $V$ is well defined.  Following the lines of \cite[Lemma 3.1]{P07} 
we state a first regularity result on $V$ that will be used hereafter  to set the dynamic programming principle.

\begin{Lemma}\label{RegularityV}
The value function $V(.,.;p_0)$ defined by \eqref{pb:optprotection}  is continuous for all $p_0\in\{0,1\}$.
More precisely, there exists some positive constant $C$ such that for any couple of any conditions $(s_0,i_0)$,  $(s'_0,i'_0)$
$$|V(s_0,i_0;p_0)-V( s'_0, i'_0;p_0)|\leq C\left(|s_0- s'_0|+|i_0-i'_0|\right).$$
\end{Lemma}

\subsection{Dynamic programming, Viscosity Solutions and value function properties}
In this part, we state the dynamic programming principle which is a well-known property in stochastic optimal control and  allows us to derive the PDE properties of the value function.

\subsubsection{Dynamic Programming Principle and Viscosity solutions}
 
 Following \cite{P07}, the dynamic programming principle is formulated in our context in this way:\\
 
For any initial state $(s_0, i_0)$  and initial regime $(a, p_0 )$
\begin{eqnarray}\label{dpp}
 V(s_0,i_0;p_0)&=&\underset{\alpha \in \mathcal A^p}{\inf} \mathbb E\Big[\int_0^{\theta}e^{-\delta t} (c_I I^{\alpha,\tilde \alpha}_t + f(S^{\alpha,\tilde \alpha}_t,p_t))dt +e^{-\delta \theta}V(S_\theta^{\alpha,\tilde \alpha},I_\theta^{\alpha,\tilde \alpha};p_\theta)\nonumber\\
&  &+\sum_{n\geq 1} e^{-\delta \tau_n} \mathbf{1}_{\{\tau_n\leq \theta\}}  g_{p_{\tau_{n-1}}, p_{\tau_n}}(S^{\alpha,\tilde \alpha}_{ \tau_n },I^{\alpha,\tilde \alpha}_{ \tau_n })\Big] ,
 \end{eqnarray}

where $\theta$ is any stopping time, possibly depending on $\alpha \in \mathcal A^p$.\\

The dynamic programming principle combined with the notion of viscosity solutions are known to be a general and powerful tool for characterizing the value function of a stochastic control problem via a PDE representation.\\

 
 
We define now the operator $\mathcal L^{a,p}$ by
 \[\mathcal L^{a,p} v(s,i;a,p)= (\rho(1-s-i) -s(p \kappa +a \nu+\beta i)) \partial_s v +( a \nu s -\gamma i + \beta si) \partial_i v +\frac{\sigma^2}2 s^2i^2( \partial_{ss}  v+ \partial_{ii}  v -2 \partial_{is} v ). \]
 
From now on we define $\overline p$ by $\bar p=0$ if $p=1$, or $\bar p=1$ if $p=0$.  We thus introduce the following system of variational inequalities together with the switching and the continuations regions, for any $p,a\in \{0,1\}$
   \begin{align}\label{eq:VPDE}
& \min [ - \delta v(s,i; a,p) + \mathcal L^{a,p}v (s,i;a,p) + c_Ii+ f(s, p), v(s,i; a,\bar p)+g_{p,\bar p}(s,i)- v(s,i; a,p)]=0,\\
& \quad \quad \text{ on the set  }  \mathcal D:=\{(s,i)\in [0,1]^2,\; s+i \leq 1\}. \nonumber
 \end{align}
 

 \begin{Remark}\label{rem:s0}
Note that for a SIR model (that is for the special case  $\rho=0$), $s_0=0$ implies $S_t=0$ at any time $t$ and consequently $I_t=i_0 e^{-\gamma t},$ for any time $t\geq 0$. In this case, the system of variational inequalities \eqref{eq:VPDE} admits the initial value 
$$v(0,i;a,p)=\frac{c_I i}{\delta+\gamma}.$$
\end{Remark}

 Given a fixed value $a$ in $\{0,1\}$ of the hacker's strategy, we define the following switching and continuation regions for any $p\in \{0,1\}$: 
\begin{itemize}
\item Switching region from $p$ to $\overline p$:
\begin{equation}\label{defs}
\mathcal S^{a}_{p,\overline p}:=\{(s,i)\in \mathcal D,\; v(s,i;a,p)=v(s,i;a,\overline p)+g_{p,\overline p}(s,i)\};
\end{equation}
\item Continuation region in $p$: 
\begin{equation}\label{defc}
\mathcal C^a_{p}:=\{(s,i)\in \mathcal D,\; v(s,i;a,p)<v(s,i;a,\overline p)+g_{p,\overline p}(s,i)\}.
\end{equation}
\end{itemize}
For sake of completeness, we recall the definition of viscosity solutions to the  system of   variational inequalities \eqref{eq:VPDE}.
\begin{Definition}
\leavevmode\par
\begin{itemize}
    \item [$(i)$] For any $a,p\in\{0,1\}$, a continuous function $v$ on $\mathcal D$ is called a viscosity supersolution (resp. subsolution ) to the system of variational inequalities \eqref{eq:VPDE} if for any $(s,i)\in [0,1]^2$ and  $\varphi\in C^{2,2}(\mathcal D,\mathbb R)$ such that
$\varphi(s,i)= v(s,i;a,p)$ and $(s,i)$ is a local minimum (resp. maximum) of $\varphi-v$, we have
\begin{equation*}
\hspace*{-10mm}
\begin{aligned}\min\{- \delta \varphi(s,i)  + \mathcal L^{a,p}\varphi(s,i)  + c_I i + f(s,p) ,-v(s,i;a,p)+v(s,i;a,\bar p)+ g_{p,\bar p}(s,i)\}&\geq& 0
\\(resp. &\leq& 0).
\end{aligned}
\end{equation*}
    \item[$(ii)$] $v$ is a viscosity solution if it is both a viscosity supersolution and subsolution.
\end{itemize}
\end{Definition}
The value function \eqref{pb:optprotection} can  be characterized as follows.
\begin{Theorem} 
\label{thviscosity}
For each $p\in\{0,1\}$, the value function $V$ is a continuous viscosity solution on $\mathcal D$ to the variational inequality \eqref{eq:VPDE}. 
\end{Theorem}
{\bf Proof of the supersolution property:}\\
First, for any $(s,i,p)\in \mathcal D\times\{0,1\}$, we obtain, thanks to \eqref{dpp}, and by choosing the immediate switching control $\tau_1=0$, $p_{\tau_1}=\bar p$, $\tau_n=\infty, n\geq 2$ and $\theta=0$
\begin{equation}\label{supersol}
V(s,i;p)\leq V(s,i;\bar p)+ g_{p,\bar p}(s,i).
\end{equation}
Now, let $\varphi\in C^{2,2}(\mathcal D,\mathbb R)$ such that
\begin{equation}\label{superslmax}
\varphi(s,i)-V(s,i;p)=\underset{\mathcal D}{\min}(\varphi-V(.,.;p))=0
\end{equation}
It remains to show that
\begin{equation}\label{supersolbis}
- \delta \varphi(s,i)  + \mathcal L^{a,p}\varphi(s,i) + c_Ii+ f(s, p) \geq 0.
\end{equation}
By using the dynamic programming principle \eqref{dpp} for $\theta= h$ and taking the no-switching control $\tau_n=\infty$, we get
\begin{equation}\label{condsupersol}
 V(s,i;p)\leq\mathbb E\left[\int_0^{\theta}e^{-\delta t} (c_I I^{\alpha,  \tilde \alpha}_t + f(S^{\alpha,  \tilde \alpha}_t,p_t))dt +e^{-\delta \theta}V(S_\theta^{\alpha,  \tilde \alpha},I_\theta^{\alpha,  \tilde \alpha};p_\theta)\right].
  \end{equation}
Applying Itô's formula to $e^{-\delta t}\varphi(S_t^{\alpha,  \tilde \alpha},I_t^{\alpha,  \tilde \alpha})$ between $0$ and $\theta$ and since\\ $ \left( \partial_s \varphi +  \partial_i \varphi \right) (S_t^{\alpha,  \tilde \alpha},I_t^{\alpha,  \tilde \alpha}) \sigma S_t^{\alpha,  \tilde \alpha}I_t^{\alpha,  \tilde \alpha}$ is bounded, we obtain
\begin{eqnarray}\label{supersolbisbis}
& &\frac{1}{h}\mathbb E\left[\int_0^{\theta}e^{-\delta t}\left(- \delta \varphi(S_t^{\alpha,  \tilde \alpha},I_t^{\alpha,  \tilde \alpha}) + \mathcal L^{a,p}\varphi(S_t^{\alpha,  \tilde \alpha},I_t^{\alpha,  \tilde \alpha})  + c_I I_t^{\alpha,  \tilde \alpha}+ f(S_t^{\alpha,  \tilde \alpha}, p_t)  \right)dt\right]\nonumber\\
&\quad &=\frac{1}{h}\mathbb E\left[e^{-\delta \theta}\varphi(S_\theta^{\alpha,  \tilde \alpha},I_\theta^{\alpha,  \tilde \alpha})-\varphi(s,i)+\int_0^{\theta}e^{-\delta t}(c_I I^{\alpha,  \tilde \alpha}_t + f(S^{\alpha,  \tilde \alpha}_t,p_t))dt\right]\geq 0,\nonumber\\
\end{eqnarray}
where we have used for the last line inequalities \eqref{superslmax} and \eqref{condsupersol}.\\
From the dominated convergence theorem, this yields by sending $h$ to zero
$$- \delta \varphi(s,i)  + \mathcal L^{a,p}\varphi(s,i) + c_I i + f(s,p) \geq 0.$$
By combining with \eqref{supersol}, we get
\begin{equation}\label{supersolproperty}
\min\{- \delta \varphi(s,i)  + \mathcal L^{a,p}\varphi(s,i)  + c_I i + f(s,p) ,-V(s,i;p)+V(s,i;\bar p)+ g_{p,\bar p}(s,i)\}\geq 0.
\end{equation}
\fproof\\
{\bf Proof of the subsolution property}:\\
Let $(s,i,p)\in [0,1] \times [0,1]\times\{0,1\}$ and  $\varphi\in C^{2,2}(\mathcal D,\mathbb R)$ such that
\begin{equation}\label{supersolmax}
\varphi(s,i)- V(s,i;p)=\underset{\mathcal D}{\max}(\varphi-V(.,.;p))=0
\end{equation}
We argue by contradiction by assuming in the contrary that
\begin{eqnarray*}
- \delta \varphi(s,i)  + \mathcal L^{a,p}\varphi(s,i) + c_I i+ f(s,p)> 0\,,\label{contradiction}\\
\mbox{ and } \quad -V(s,i;p)+V(s,i;\bar p)+ g_{p,\bar p}(s,i)>0\,.\label{contradictionbis}
\end{eqnarray*}
By continuity of $V$, $\varphi$ and its derivatives, there exists some $ \epsilon >0$ such that
\begin{eqnarray}
- \delta \varphi(s',i')  + \mathcal L^{a,p}\varphi(s',i') + c_I i' + f(s',p) &\geq &\epsilon\,,\, \forall (s',i')\in B_\epsilon(s,i)\nonumber\\
\label{inequalitycontradiction}\\
-V(s',i';p)+V(s',i';\bar p)+ g_{p,\bar p}(s,i)&\geq&\epsilon\,,\, \forall (s',i')\in B_\epsilon(s,i).\nonumber\\
\label{inequalitycontradictionbis}
\end{eqnarray}
For any $\alpha=(\tau_n)_{n\geq 1}\in\mathcal A^p$, consider the exit time $\tau_{\epsilon}:=\inf\{t\geq 0, (S_t^{\alpha,  \tilde \alpha},I_t^{\alpha,  \tilde \alpha})\notin B_\epsilon(s,i)\}$.
By applying Itô's formula to $e^{-\delta t}\varphi(S_t^{\alpha,  \tilde \alpha},I_t^{\alpha,  \tilde \alpha})$ between $0$ and $\theta=\tau_1\wedge\tau_\epsilon$, we have by noting that before $\theta$, $(S^{\alpha,  \tilde \alpha},I^{\alpha,  \tilde \alpha})$ stays in regime $p$ and in the ball $B_\epsilon(s,i)$:
\begin{eqnarray}\label{subsol}
V(s,i;p)\!&=&\!\varphi(s,i)\!\!=\!\!\mathbb E\left[\!e^{-\delta \theta}\varphi(S_\theta^{\alpha,  \tilde \alpha},I_\theta^{\alpha,  \tilde \alpha})\!\!+\!\!\int_0^{\theta}\!\!e^{-\delta t}(\delta \varphi(S_t^{\alpha,  \tilde \alpha},I_t^{\alpha,  \tilde \alpha}) \!-\! \mathcal L^{a,p}\varphi(S_t^{\alpha,  \tilde \alpha},I_t^{\alpha,  \tilde \alpha}))  dt\right]\nonumber\\
\!&\leq&\!\mathbb E\left[\!e^{-\delta \theta}V(S_\theta^{\alpha,  \tilde \alpha},I_\theta^{\alpha,  \tilde \alpha};p)\!\!+\!\!\int_0^{\theta}\!e^{-\delta t}(\delta \varphi(S_t^{\alpha,  \tilde \alpha},I_t^{\alpha,  \tilde \alpha}) \!-\! \mathcal L^{a,p}\varphi(S_t^{\alpha,  \tilde \alpha},I_t^{\alpha,  \tilde \alpha}))dt\right]\nonumber\\
\end{eqnarray}
Now, since $\theta=\tau_1\wedge\tau_\rho$, we have
\begin{eqnarray*}
e^{-\delta \theta}V(S_\theta^{\alpha,  \tilde \alpha},I_\theta^{\alpha,  \tilde \alpha}; p_\theta)
&+&\!\sum_{n\geq 1} e^{-\delta \tau_n} \mathbf{1}_{\{\tau_n\leq \theta\}}  g_{p_{\tau_{n-1}}, p_{\tau_n}}(S_{\tau_n}^{\alpha,  \tilde \alpha},I_{\tau_n}^{\alpha,  \tilde \alpha})\\&\geq&\!\!e^{-\delta \tau_1}(V(S_{\tau_1}^{\alpha,  \tilde \alpha},I_{\tau_1}^{\alpha,  \tilde \alpha}; \bar p)+  g_{p,\bar p })\mathbf 1_{\tau_1\leq \tau_\epsilon}\!+\!e^{-\delta \tau_\epsilon}V(S_{\tau_\epsilon}^{\alpha,  \tilde \alpha},I_{\tau_\epsilon}^{\alpha,  \tilde \alpha};p)\mathbf 1_{\tau_\epsilon < \tau_1}\\
\!\!&\geq&\!\!e^{-\delta \tau_1}(V(S_{\tau_1}^{\alpha,  \tilde \alpha},I_{\tau_1}^{\alpha,  \tilde \alpha};p)+  \epsilon)\mathbf 1_{\tau_1\leq \tau_\epsilon}\!+\!e^{-\delta \tau_\epsilon}V(S_{\tau_\epsilon}^{\alpha,  \tilde \alpha},I_{\tau_\epsilon}^{\alpha,  \tilde \alpha};p)\mathbf 1_{\tau_\epsilon < \tau_1}\\
\!\!&=&\!\!e^{-\delta \theta}V(S_{\theta}^{\alpha,  \tilde \alpha},I_{\theta}^{\alpha,  \tilde \alpha};p)+ \epsilon e^{-\delta \tau_1}\mathbf 1_{\tau_1\leq \tau_\epsilon},
\end{eqnarray*}
where the inequality follows from \eqref{inequalitycontradictionbis}. By plugging into \eqref{subsol} and using \eqref{inequalitycontradiction}, we get
\begin{eqnarray}\label{subsolbis}
V(s,i;p)\!&\leq&\!\mathbb E\Big[\!\int_0^{\theta}\!e^{-\delta t}(\!c_I I^{\alpha,  \tilde \alpha}_t \!+ \!f(S^{\alpha,  \tilde \alpha}_t,p_t))dt+e^{-\delta \theta}V(S_\theta^{\alpha,  \tilde \alpha},I_\theta^{\alpha,  \tilde \alpha};p_\theta)\nonumber\\
&+&\sum_{n\geq 1} e^{-\delta \tau_n} \mathbf{1}_{\{\tau_n\leq \theta\}} g_{p_{\tau_{n-1}}, p_{\tau_n}}(S_{\tau_n}^{\alpha,  \tilde \alpha},I_{\tau_n}^{\alpha,  \tilde \alpha})\Big]
- \epsilon \,\mathbb E\left[\!\int_0^{\theta}\!e^{-\delta t}dt+e^{-\delta \tau_1}\mathbf 1_{\tau_1\leq \tau_\epsilon}\right].\nonumber\\
\end{eqnarray}
On the other hand we note from the result shown in the proof of Theorem 3.1 \cite{P07} that there  exists some positive constant $c_0>0$ such that
$$ \mathbb E\left[\!\int_0^{\theta}\!e^{-\delta t}dt+e^{-\delta \tau_1}\mathbf 1_{\tau_1\leq \tau_\epsilon}\right]\geq c_0, \quad \forall \alpha\in\mathcal A^p.$$
Finally, by including this last inequality (uniform in $\alpha$) into \eqref{subsolbis}, we obtain :
\begin{eqnarray*}
V(s,i;p)\!&\leq&\!\inf_{\alpha \in \mathcal A^p} \mathbb E\Big[\!\int_0^{\theta} e^{-\delta t}(c_I I^{\alpha,  \tilde \alpha}_t \!+ \!f(S^{\alpha,  \tilde \alpha}_t,p_t))dt+e^{-\delta \theta}V(S_\theta^{\alpha,  \tilde \alpha},I_\theta^{\alpha,  \tilde \alpha}; p_\theta)\!\\
& &+\!\sum_{n\geq 1} e^{-\delta \tau_n} \mathbf{1}_{\{\tau_n\leq \theta\}}  g_{p_{\tau_{n-1}}, p_{\tau_n}}(S_{\tau_n}^{\alpha,  \tilde \alpha},I_{\tau_n}^{\alpha,  \tilde \alpha})\Big]
- \epsilon \,c_0,
\end{eqnarray*}
which is in contradiction with dynamic programming principle \eqref{dpp}.
\fproof\\
In a manner similar to Lemma 4.1 in \cite{PH07}, we state that $V$ is a viscosity solution of the variational system on $\mathcal D$, and a regular solution of a PDE on each continuation region (for a constant strategy of the hacker), satisfying smooth fit condition on the boundary. 

\begin{Lemma}
 For  a fixed  cluster owner's strategy  $p\in\{0,1\}$ in a fixed hacking environment  $a\in\{0,1\}$, the value function $V(.,.;p)$ is smooth $C^{2,2}$ on $\mathcal C^a_{p}$, and satisfies in a classical sense the following PDE: 
\begin{equation}\label{PDE:vC}  -\delta v(s,i; a,p) + \mathcal L^{a,p} v(s,i;a,p)+ c_I i + f(s,p)=0,  \quad  (s,i)\in\mathcal C^a_{p}.
\end{equation}

\end{Lemma}

We can also derive the smooth-fit property of the value function 
$V$ through the boundaries of the switching regions by following Theorem 4.1 in \cite{P07}.
\begin{Lemma}
For all $p\in\{0,1\}$, and constant hacker's strategy $a$, the value function $V(.,.; p)$ is continuously differentiable on $\mathcal D$. Moreover, at $(s,i)\in\mathcal S_{p,\bar p}^a$, we have 
$$\partial_s V(s,i; p)=\partial_s V(s,i; \bar p) \quad \text{and} \quad \partial_i V(s,i; p)=\partial_i V(s,i;\bar p).$$
\end{Lemma}
\begin{Remark}
 The uniqueness of the system of variational inequalities \eqref{eq:VPDE} has been studied in \cite{el2010optimal}. In particular, the same arguments as those used in Theorem~5 of \cite{el2010optimal} can be applied to show that the viscosity solution to \eqref{eq:VPDE} is unique within the class of continuous functions with linear growth on $\mathcal{D}$.

\end{Remark}

\subsubsection{Link with a system of reflected BSDE}

In this section, we give the probabilistic representation of the value function $V$ solving \eqref{pb:optprotection} as a system of reflected BSDE with infinite horizon. 
We introduce the following spaces. 
\begin{itemize}
\item $\mathcal S^2(\mathbb R)$  is the set of $\mathbb R$-valued adapted and c\`adl\`ag processes $(Y_t)_{t\geq 0}$ such that $$\mathbb E[\sup_{t\geq 0}|Y_t|^2]<\infty,$$
\item $\mathcal H^2(\mathbb R)$ is the set of $\mathbb R$-valued, progressively measurable processes $(Z_t)_{t\geq 0}$ such that
 $$\mathbb E\Big[\int_0^\infty|Z_t|^2 dt\Big]<\infty,$$
 \item $\mathcal K^2(\mathbb R)$ is the set of non-decreasing processes $K$ in $\mathcal S^2(\mathbb R)$ with $K_0 = 0$.
\end{itemize}

We set for any $(a,p)\in \{0,1\}\times \{0,1\}$

\begin{equation}\label{relfectedBSDE}\begin{cases}
 e^{-\delta t}Y_t^{a,p}=\displaystyle\int_t^\infty e^{-\delta s}[c_I I_t^{\alpha,\tilde\alpha} + f(S_t^{\alpha,\tilde\alpha},p_t)] ds -\int_t^\infty e^{-\delta s}Z_s^{a,p}dW_s +K_\infty^{a,p}-K_t^{a,p}\\
 \lim\limits_{t\to\infty}e^{-\delta t}  Y_t^{a,p}=0\\
 Y_t^{a,p}\leq Y_t^{a,\overline p}+g_{p,\overline p}(S_{t}^{\alpha,  \tilde \alpha},I_{t}^{\alpha,  \tilde \alpha})\\
\displaystyle \int_0^\infty e^{-\delta s}   [Y_s^{a,p}- (Y_s^{a,\overline p}+g_{p,\overline p}(S_{t}^{\alpha,  \tilde \alpha},I_{t}^{\alpha,  \tilde \alpha}))]dK_s^{a,p}=0.
\end{cases}\end{equation}
Reflected BSDEs with finite horizon have been widely investigated in the literature, see for example \cite{el1997reflected,hu2010multi,chassagneux2011note}, and \cite{hamadene1999infinite,aid2014probabilistic,el2010optimal,aazizi2013optimal} for the link with switching problem. Applying Proposition 3 and Theorem 5 in \cite{el2010optimal}, we directly get the following result. 
\begin{Proposition}\label{RBSDE}
  The reflected BSDE \eqref{relfectedBSDE} admits a unique solution $(Y^{a,p},Z^{a,p},K^{a,p})\in \mathcal S^2(\mathbb R)\times \mathcal H^2(\mathbb R)\times \mathcal K^2 (\mathbb R)$ for any regime $(a,p)$ and for any initial regime $(a_0,p_0)\in \{0,1\}\times \{0,1\}$, we have
  \[ Y_0^{a_0,p_0}=V^{\tilde\alpha}(s_0,i_0;p_0).\]
\end{Proposition}
\begin{Remark}
Proposition \ref{RBSDE} allows us to characterize the solution of \eqref{pb:optprotection} as the unique solution of the system of infinite horizon coupled RBSDE \eqref{relfectedBSDE}. Note that both the obstacle and the Skorokhod condition depend on the solution of the RBSDE $Y^{a,\bar p}$. For this reason, we prefer to solve the problem numerically using a PDE-based approach.
\end{Remark}
\subsection{Verification argument }

\subsubsection{Verification procedure}

 In this section, we formally prove that a \textit{smooth}\footnote{Smooth refers here to the set of function $v(\cdot,a,p):[0,1]\times [0,1]\longrightarrow \mathbb R$ for $(a,p)\in\{0,1\}\times \{0,1\}$ twice continuously differentiable function on the continuation region $\mathcal C_a^p$ and continuous on the switching region $\mathcal S_{p,\bar p}^a$. We denote this space by $C^2(\mathcal C_a^p)\cap C(\mathcal S_{p,\bar p}^a)$.} solution to the variational inequalities \eqref{eq:VPDE} provides a solution to the optimization problem \eqref{pb:optprotection}. This section follows \cite{bouchard2011optimal,belak2017general} extended to one strategic player in a random environment and \cite{basei2022nonzero,aid2021impulse} to adversarial environment adapted to our problem.
 Assume that there is a family of smooth functions $\{v(.,.;a,p), \, a \in\{0,1\}, \, p\in\{0,1\}\}$ which solves \eqref{eq:VPDE}.
 We fix the initial data $(s_0,i_0)$ and  the initial regime $p_0$.  The hacker's strategy $\tilde \alpha$ is exogenous: the 
 state of attack at  time $t$ is 
 $a_t = \displaystyle \sum_{n\geq 0}  \mathbf 1_{\tilde\tau_{2n+1-a_0}  \leq t< \tilde\tau_{2n+2-a_0}}$.


 



 We define $(S^*, I^*, R^*)$ the solution of the SDEs system \eqref{systdynamics} characterized  by the switching  sequence  $\alpha^*:=(\tau^*_n)_{n\geq 0}$ that contains  both the switching times of the hacker (that is $\tilde \alpha$) and the optimal switching times of the cluster owner (that  is $\hat \alpha$) defined by induction as follows.

 \textit{Initialization. } Starting at $a_0$ for the hacker and $p_0$ for the cluster owner, we set
 
\[\tau^*_1=\inf\{t>0,\;  v^{}( S^*_t, I^*_t; a_0,p_{0})= v^{}(S^*_t, I^*_t; a_0,\bar p_{0})+g_{p_{0},\bar p_{0}}\}\wedge \tilde\tau_1.\]

\textit{Induction.} \text{For $n> 1$. }
$$
  \tau^*_{n}= \inf\{t>  \tau^*_{n-1} ,\;  v^{ }( S^*_t, I^*_t;a_{\tau^*_{n-1}} ,p_{\tau^*_{n-1}})= v^{}( S^*_t, I^*_t;a_{\tau^*_{n-1}},\bar p_{\tau^*_{n-1} })+g_{p_{\tau^*_{n-1}},\bar p_{\tau^*_{n-1} }}\} \wedge \tilde\tau(\tau_{n-1}^*),
$$
with $\tilde\tau(\tau_{n-1}^*)=\min_{k\ge1}\{\tilde\tau_k:\ \tilde\tau_k\ge \tau_{n-1}^*\}.$\\

Before proceeding with  the verification theorem,
we recall the continuation and switching regions as defined in \eqref{defs} and \eqref{defc}.
\begin{itemize}
    \item Switching region from $p$ to $\overline p$: \;
   $\mathcal S^{a}_{p,\overline p}:=\{(s,i):\; v(s,i;a,p)=\mathcal M v(s,i; a, p)\};$
    \item Continuation region in $p$: \;
  $  \mathcal C^a_{p}:=\{(s,i):\; v(s,i;a,p)<\mathcal M v(s,i; a, p)\}$
\end{itemize}
where $\mathcal M v(s,i; a, p) := v(s,i; a, \overline p) + g_{p,\overline p}(s,i)$.

\begin{Theorem}[Verification and Construction of Optimal Strategy]\label{thm:verification}
Fix an initial regime $(a_0,p_0)\in\{0,1\}^2$. For each $(a,p)\in\{0,1\}^2$, suppose that there exists a function $v(\cdot,\cdot; a,p)\in C^{2}(\mathcal C^a_p)\cap C(\mathcal S^a_{p,\overline p})$ satisfying the following Quasi-Variational Inequalities for all $(s,i,a,p)\in [0,1]^2\times \{0,1\}^2$:
\begin{equation*}
\max \Biggl\{ 
-\mathcal L^{a,p} v(s,i; a, p) + \delta v(s,i; a, p) - (c_I i + f(s,p)) , \quad 
 v(s,i; a, p) - \mathcal M v(s,i; a, p) 
\Biggr\} = 0,
\end{equation*}
Define the sequence of event times $\tau_n^*$ starting with $\tau_0^*=0$ as follows:
\[
\tau_n^* := \hat\tau_n \wedge \tilde\tau(\tau_{n-1}^*),
\]
where
\begin{itemize}
    \item $\hat\tau_n := \inf\{t > \tau_{n-1}^* : (S_t, I_t) \in \mathcal S^{a^*_{n-1}}_{p^*_{n-1},\overline p^*_{n-1}}\}$ is the defender's planned switching time after the $(n-1)^{\text{th}}$ regime switching.\footnote{For the sake of simplicity and to alleviate the notations, we set $(a^*_{n-1}, p^*_{n-1},\overline p^*_{n-1}):=(a_{\tau_{n-1}^*}, p_{\tau_{n-1}^*},\overline p_{\tau_{n-1}^*})$. }
    \item $\tilde\tau(\tau_{n-1}^*) := \min_{k \ge 1} \{ \tilde\tau_k : \tilde\tau_k > \tau_{n-1}^* \}$ is the next random hacker switching time after the $(n-1)^{\text{th}}$ regime switching.
\end{itemize}

The regimes update as follows:
\begin{itemize}
    \item If $\tau_n^* = \hat\tau_n < \tilde\tau(\tau_{n-1}^*)$, we set $
    p_{\tau_n^*} = \overline p_{\tau_{n-1}^*}, \quad a_{\tau_n^*} = a_{\tau_{n-1}^*}.$
    \item If $\tau_n^* = \tilde\tau(\tau_{n-1}^*) \le \hat\tau_n$, we set $
    a_{\tau_n^*} = \overline a_{\tau_{n-1}^*}, \quad p_{\tau_n^*} = p_{\tau_{n-1}^*}.$
\end{itemize}

Then, the subsequence of times $\hat \alpha = (\tau_n^* \mid \tau_n^* = \hat\tau_n)_n$ is an optimal protection strategy, and $V(s_0,i_0; a_0,p_0) = v(s_0,i_0; a_0,p_0)$.
\end{Theorem}

\begin{proof}
We prove the result by induction on the number of cluster owner interventions $n$.
We define the claim $\mathcal P_n$ as:
\begin{equation*}
\mathcal P_n:\; v(S_0,I_0;a_0,p_0) = \mathbb E\left[ \int_{0}^{\hat\tau_{n}} e^{-\delta s} (c_I I_s+f(S_s,p_s)) ds +\mathfrak S^n  + e^{-\delta \hat\tau_n}v(S_{\hat\tau_n},I_{\hat\tau_n}; a_{\hat\tau_n},p_{\hat\tau_n}) \right], 
\end{equation*}
with \[\mathfrak S^n:=\sum_{j=1}^n e^{-\delta \hat\tau_j} g_{p_{\hat\tau_{j-1}},p_{\hat\tau_{j}}}(S_{ \hat\tau_j },I_{ \hat\tau_j }) \mathbf 1_{\tau_j^* = \hat\tau_j}.\]

\noindent\textit{ Base Case ($n=1$).}
We consider the stopping time $\tau^*_1 = \hat\tau_1 \wedge \tilde\tau_1$. On the interval $[0, \tau^*_1)$, the parameters $(a_0, p_0)$ are constant. We apply It\^o's formula to the process $Y_t = e^{-\delta t}v(S_t, I_t; a_0, p_0)$. 
Since $v$ satisfies the constraint $-\mathcal L^{a_0,p_0} v + \delta v = c_I i + f$ in the continuation region $\mathcal C^{a_0}_{p_0}$, we obtain:
\begin{align*}
e^{-\delta \tau^*_1}v(S_{\tau^*_1},I_{\tau^*_1}; a_0, p_0) &= v(S_0,I_0; a_0, p_0) - \int_{0}^{\tau^*_1} e^{-\delta s}(c_I I_s + f(S_s, p_0)) ds \\
&\quad + \int_{0}^{\tau^*_1} e^{-\delta s} [\partial_I - \partial_S]v(S_s, I_s; a_0, p_0) \sigma S_s I_s dW_s.
\end{align*}

Taking expectations we get:
\begin{equation}\label{eq:ito_base}
v(S_0,I_0; a_0, p_0) = \mathbb E\left[ \int_{0}^{\tau^*_1} e^{-\delta s}(c_I I_s + f(S_s, p_0)) ds + e^{-\delta \tau^*_1}v(S_{\tau^*_1},I_{\tau^*_1}; a_0, p_0) \right].
\end{equation}

To prove $\mathcal{P}_1$, we decompose the terminal term on the disjoint events $\{\hat\tau_1 < \tilde\tau_1\}$ (cluster owner acts first) and $\{\tilde\tau_1 \le \hat\tau_1\}$ (hacker acts first).

\begin{enumerate}
    \item[(i)] On the set $\{\hat\tau_1 < \tilde\tau_1\}$:
    We have $\tau^*_1 = \hat\tau_1$. By definition of the switching region $\mathcal S^{a_0}_{p_0,\overline p_0}$, the cluster owner switches $p$ optimally:
    \[
    v(S_{\hat\tau_1}, I_{\hat\tau_1}; a_0, p_0) = v(S_{\hat\tau_1}, I_{\hat\tau_1}; a_0, \overline{p}_0) + g_{p_0, \overline{p}_0} (S_{ \hat\tau_1 },I_{ \hat\tau_1 }).
    \]

    \item[(ii)] On the set $\{\tilde\tau_1 \le \hat\tau_1\}$:
    We have $\tau^*_1 = \tilde\tau_1$. At this instant, the hacker switches $a_0$ to $\overline{a}_0$. This switch is exogenous and not strategic for the cluster owner since the switching of the environment (hacker strategy) appears only in the SIRS system. Note that we restrict to the case where the hacker switches only once before $\hat\tau_1$. The proof would be similar if we assume that there exists $k$ such that $0<\tilde \tau_1<\dots<\tilde\tau_k<\hat\tau_1$. From It\^o's formula between each change of strategy of the hacker until $\hat\tau_1$ is reached together with the continuity of the value function, the integral terms accumulate into a single integral between 0 and $\hat\tau_1$.
\end{enumerate}

Combining (i) and (ii) into Equation \eqref{eq:ito_base} we get:
\begin{align*}
    v(S_0, I_0;a_0,p_0) &= \mathbb E\left[ \int_{0}^{\hat\tau_1} e^{-\delta s}(c_I I_s + f(S_s,p_0)) ds \right] \\
    &\quad + \mathbb E\left[ e^{-\delta \hat\tau_1} \left( v(S_{\hat\tau_1}, I_{\hat\tau_1}; a_{\hat\tau_1}, \overline{p}_0) + g_{p_0, \overline{p}_0} \right) \mathbf{1}_{\{\hat\tau_1 < \tilde\tau_1\}} \right] \\
    &\quad + \mathbb E\left[ e^{-\delta \tilde\tau_1} v(S_{\tilde\tau_1}, I_{\tilde\tau_1}; \overline{a}_0, p_0) \mathbf{1}_{\{\tilde\tau_1 \le \hat\tau_1\}} \right]\\
    &= \mathbb E[ \int_{0}^{\hat\tau_1} e^{-\delta s}(c_I I_s + f(S_s,p_0)) ds+  e^{-\delta \hat\tau_1} ( v(S_{\hat\tau_1}, I_{\hat\tau_1}; a_{\hat\tau_1},p_{\hat\tau_1}) + g_{p_0, \overline{p}_0} (S_{ \hat\tau_1 },I_{ \hat\tau_1 })  \mathbf{1}_{\{\hat\tau_1=\tau^*_1 \}})]
\end{align*}
We thus obtain the base case $\mathcal{P}_1$.

\noindent\textit{Induction.}
We assume that $\mathcal P_j$ holds for any $j\leq n$ for a fixed $n\geq 1$. We need to prove $\mathcal P_{n+1}$. 
We analyze the system's evolution over the interval $(\hat\tau_n, \hat\tau_{n+1}]$. In this interval, the owner's control $p_s$ remains constant at $p_{\hat\tau_n}$. Following the same lines that the base case, we get
\begin{align*}
e^{-\delta \hat\tau_n}v(S_{\hat\tau_n}, I_{\hat\tau_n}; a_{\hat\tau_n}, p_{\hat\tau_n}) = \mathbb{E}\Bigg[ &\int_{\hat\tau_n}^{\hat\tau_{n+1}} e^{-\delta s} \left(c_I I_s + f(S_s, p_{\hat\tau_n})\right) ds \\
&+ e^{-\delta \hat\tau_{n+1} } \left( g_{p_{\hat\tau_n}, p_{\hat\tau_{n+1}}}(S_{\hat\tau_n},I_{\hat\tau_n})\mathbf 1_{\hat\tau_n=\tau^*_n} + v(S_{\hat\tau_{n+1}}, I_{\hat\tau_{n+1}}; a_{\hat\tau_{n+1}}, p_{\hat\tau_{n+1}}) \right) \Bigg| \mathcal{F}_{\hat\tau_n} \Bigg].
\end{align*}
Using the induction assumption together with the law of iterated conditional expectations, we get by summation of the integrals

\begin{align*}
v(S_0, I_0; a_0, p_0) = \mathbb{E}\Bigg[ &\int_{0}^{\hat\tau_n} e^{-\delta s} \left(c_I I_s + f(S_s, p_s)\right) ds + \sum_{k=1}^n e^{-\delta \hat\tau_k} g_{p_{\hat\tau_{k-1}}, p_{\hat\tau_k}}(S_{\hat\tau_k},I_{\hat\tau_k})\mathbf 1_{\hat\tau_k=\tau^*_k} \\
&+ \int_{\hat\tau_n}^{\hat\tau_{n+1}} e^{-\delta s} \left(c_I I_s + f(S_s, p_{\hat\tau_n})\right) ds + e^{-\delta \hat\tau_{n+1}} g_{p_{\hat\tau_n}, p_{\hat\tau_{n+1}}}(S_{\hat\tau_n},I_{\hat\tau_n})\mathbf 1_{\hat\tau_n=\tau^*_n} \\
&+ e^{-\delta \hat\tau_{n+1}} v(S_{\hat\tau_{n+1}}, I_{\hat\tau_{n+1}}; a_{\hat\tau_{n+1}}, p_{\hat\tau_{n+1}}) \Bigg],
\end{align*}
and so $\mathcal P_{n+1}$ is satisfied.\vspace{0.3em}

\textit{Conclusion.} By taking $n\to+\infty$ since $v$ is continuous on a bounded domain with $\hat\tau_n\to+\infty$, we get 

\[v(S_0,I_0;a_0,p_0)=\mathbb E[ \int_{0}^{\infty}( c_I I_s+f(S_s,p^*_s)) ds + \sum_{j=1}^\infty e^{-\delta \hat\tau_j} g_{p_{\hat\tau_{j-1}},p_{\hat\tau_{j}}}]. \]\qed
\end{proof}

\begin{Remark}
The value function is regime-dependent: at any time $t$, the relevant value
is $v(S^*_t,I^*_t;a_t,p_t)$, where $a_t$ and $p_t$ are the currently active regimes of attack and protection respectively. Note that $\tau^*_n$ denotes the $n^{\text{th}}$ switching time of the system while $\hat\tau_n$ denotes the first time that the cluster owner switches starting from the last $(n-1)^{\text{th}}$ switching time induced by either the defender or the hacker. 
\end{Remark}

\section{Numerical studies}

\subsection{Numerical approximation by Deep Galerkin Method}
This part is dedicated to the  development of  a  numerical algorithm to determine the optimal strategy of a cluster owner. It relies on   the Deep Galerkin method to solve numerically the PDE \eqref{eq:VPDE}. The main idea behind solving PDEs using the Deep Galerkin Method (DGM) described in the work of Sirignano and Spiliopoulos \cite{DGM} is to represent the unknown function of interest using a deep neural network. Noting that the function must satisfy a known PDE, the network is trained by minimizing losses related to the differential operator acting on the function along with any initial, terminal and/or boundary conditions the solution must satisfy. The training data for the neural network consists of different possible inputs to the function and is obtained by sampling randomly from the region on which the PDE is defined. One of the key features of this approach is the fact that, unlike other commonly used numerical approaches such as finite difference methods, it is mesh-free. Simulations indicate that the DGM may not suffer (as much as other numerical methods) from the curse of dimensionality associated with high-dimensional PDEs and PDE systems. A discussion of DGM and its applications can be found in Al-Aradi et al. (2018). On a related note, previous  works, e.g.  Hutzenthaler et al. (2019) and Huré et al. \cite{hure2020deep} emphasize  that deep learning-based algorithms overcome the curse of dimensionality in the numerical approximation of solutions for nonlinear PDEs.\\
In this section we fix the parameter $a$ and $p$ and we consider the following PDE:

\begin{equation}\label{DGM}
 - \delta v(s,i) + \mathcal L^{a,p}v (s,i) + c_Ii+ f(s, p)=0,\; \text{on }  \mathcal D.
 \end{equation}
where \[\mathcal L^{a,p} v(s,i)= (\rho(1-s-i) -s(p \kappa +a \nu+\beta i)) \partial_s v +( a \nu s -\gamma i + \beta si) \partial_i v +\frac{\sigma^2}2 s^2i^2( \partial_{ss}  v+ \partial_{ii}  v -2 \partial_{is} v ). \]
The DGM algorithm approximates $v(s,i)$ with a deep neural network $\hat v(s,i;\theta)$ where
$\theta\in\mathbb R^k$ are the neural network’s parameters. Note that the differential operators in $\mathcal L\hat v(s,i;\theta)$ 
can be calculated analytically. Construct the objective function:
\[\mathcal J( \hat v)=\left\|- \delta \hat v(s,i;\theta) + \mathcal L^{a,p}\hat v(s,i;\theta) + c_Ii+ f(s, p)\right\|^2_{\mathcal D,\nu_1}.\]
Notice that $\left\lVert\hat v(y)\right\rVert^2_{\mathcal Y,\nu}=\displaystyle\int_\mathcal Y |\hat v(y)|^2 \nu(y) dy$ where $\nu(y)$ is a positive probability density on $y\in\mathcal Y$. $\mathcal J(\hat v)$ measures how
well the function $\hat v(s,i;\theta)$ satisfies the PDE differential operator and initial condition. If $\mathcal J(\hat v) = 0$, then $\hat v(s,i;\theta)$ is a solution to the PDE \eqref{DGM}.\\
The goal is to find a set of parameters $\theta$ such that the function $\hat v(s,i;\theta)$ minimizes the error $
\mathcal J(\hat v)$. If the error $
\mathcal J(\hat v)$ is small, then $\hat v(s,i;\theta)$ will closely satisfy the PDE differential operator and initial condition. Therefore, a $\theta$ which minimizes $
\mathcal J(\hat v(.;\theta))$ produces a reduced-form model $\hat v(s,i;\theta)$ which approximates the PDE solution $v(s,i)$.
To estimate  $\theta$, one can minimize $
\mathcal J(\hat v)$ using stochastic gradient descent on a sequence  space points drawn at random from $\mathcal D$. This avoids ever forming a mesh.\\
The DGM algorithm is:
\begin{enumerate}
\item Generate random points $(s_n, i_n)$ from $\mathcal D$ and $(x_n, y_n)$ from $\{0\}\times [0,1]$ according to respective probability densities $\nu_1$ and $\nu_2$. 
\item Calculate the squared error $G(\theta_n, r_n)$ at the randomly sampled points $r_n=\{(s_n, i_n),(x_n, y_n)\}$ where
\begin{equation*}
    G(\theta_n, r_n)=\Big(- \delta \hat v(s_n,i_n;\theta) + \mathcal L^{a,p}\hat u(s_n,i_n;\theta) + c_Ii_n+ f(s_n, p)\Big)^2+\Big(\hat v(0,y_n;\theta)-\frac{c_I y_n}{\delta+\gamma}\Big)^2.
\end{equation*}
\item Take a descent step at the random point $r_n$:
$$\theta_n=\theta_{n+1}-\alpha_n \nabla_\theta G(\theta_n,r_n).$$
\item  Repeat until convergence criterion is satisfied.
\end{enumerate}

\subsection{Optimal protection under different attack scenarios}

This section illustrates Theorem \ref{thm:verification}
for the two following attacks scenarios 
 
 \begin{enumerate}
     \item Constant attack from the hacker. We assume that the hacker constantly attacks the cluster, that is $\tilde{\tau}_n=\infty$ for any $n \geq 1$ and $a_0=1$. 
     \item Random attack sequence from the hacker. Let $N$ be a Poisson process modeling the number of switches of the cyber-attack level. Then, $\tilde \tau_k$ is the $k$th event time of $N$ such that $a_t=a_0$ for $t\in[\tilde\tau_{2k},\tilde\tau_{2k+1}) $ and $a_t=\overline{a_0}$ for $t\in[\tilde\tau_{2k+1},\tilde\tau_{2k+2})$.
 \end{enumerate}
 We are studying the evolution of the SIRS systems under optimal attack on a time period of $T=30$ days (one month). We discretize the DGM algorithm with a time step
$h=0.125$ corresponding to 3 hours in a day. In each scenario, the contagion rate is 
$\beta=0.04$, the recovery rate is $\gamma=0.02$, the replacement rate is
$\rho=0.002$, the intensity of the attack is $\nu=0.05$, the volatility of the SIRS system is $\sigma=0.2$, the actualization parameter is 
$\delta=0.2$.  We start with only susceptible and no corrupted devices, $S_0=1,I_0=0$. In both attack scenarios, we first analyse the (pathwise) optimal protection strategy on a given trajectory of the $S$ and $I$ processes. We then provide statistics, based on Monte Carlo realizations, to assess the robustness of the optimal control strategy with respect to variations of the infection parameters $\beta$ (contagion rate) and $\nu$ (force of the attack). We choose to test the impact of these parameters, which characterize the infection itself, since the cluster owner is subject to these quantities and cannot control them.
Two different metrics are used to quantify the effectiveness  of the protection strategy. The first natural metric  is the total infection burden $\int_0^T I_t dt$, which represents the total number  of devices 
that have been corrupted during the period. The second one  is the value of the peak of corrupted devices, that is $\underset{0\leq t\leq T}{\max} I_t$. This peak is an important indicator of the saturation risk,  since the costs are not necessarily a linear function of  the number of infected, as highlighted in \cite{hillairet2021propagation}.  For example, if the proportion of effective devices falls below a given threshold, the enterprise’s activity may be entirely blocked, whereas partial (possibly degraded) operation can still be maintained as long as the proportion of corrupted devices does not exceed this saturation level.

\subsubsection{Scenario 1: constant attack}
We illustrate Theorem \ref{thm:verification} when the hacker attack stays in the state $a=1$, corresponding to the first scenario above. 
We choose the efficiency of the protection $\kappa= 0.03$, the marginal cost of protection is $c_V= 0.05 $ while the marginal cost of infected device is $
c_I=0.01$. The switching costs to reinforce or relax the protection, that is going from $p=0$ to $p=1$ or conversely are proportional to the value function in the current state and given by $g_{01}=0.001v(s,i,1,0)$ and $g_{10}=0.001v(s,i,1,1)$ respectively. Figure \ref{figscenario1} represents one path of $S,I$ without protection and with optimal protection  for  a fixed $\omega$, using   the DGM algorithm. Starting with $p=0$, we observe that the cluster owner let the attack spreading a little bit before reinforcing the system's protection until time $\hat\tau_1=9.289=$ 9 days and 7 hours (first green vertical line). It is explained by the switching cost to protect the system which is too high comparing to the cost of the infection at the beginning. Then, the cluster owner reinforces the protection until time $\hat\tau_2=22,374 =$22 days and 9 hours (second green vertical line). During this period, we observe that the cluster owner effectively manages the number of infected devices (represented by the yellow curve) much more efficiently than in scenarios without any protection (shown by the blue curve). As a result, the number of susceptible devices that have not yet been compromised by the attack (illustrated by the red curve) declines at a slower rate compared to the no-protection strategy (depicted by the pink curve). However, when the cost of maintaining protection becomes prohibitively high at time $\hat\tau_2=22.374$, the cluster owner decides to reduce the protection level to 
$p=0$. This optimal strategy has successfully kept the number of corrupted devices at a lower level after one month, in stark contrast to the outcomes observed under the no-protection approach.
\begin{figure}[htbp]
\centering
\includegraphics[width=0.8\textwidth]{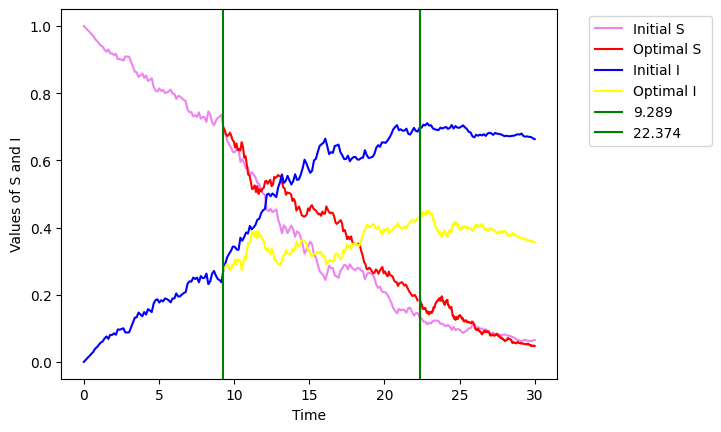}
    \caption{Optimal trajectory of $S$ and $I$ with protection and switching v.s. no protection strategy. Scenario 1.}
   \label{figscenario1}
\end{figure}

 Figures \ref{variationsbeta1} and \ref{variationsnu1}  illustrate the infection reduction effectiveness by providing the percentage of reduction of the two metrics (Peak of the infection and Total Burden) with respect to $\beta$ and $\nu$ respectively.
More precisely,  a Monte Carlo simulation procedure  is implemented within the stochastic SIRS framework. In both cases,  $200$ independent stochastic trajectories of the epidemic dynamics  are generated under two strategies: an uncontrolled baseline and an optimally controlled regime. For each trajectory, the  two  severity  indicators, namely the peak infection level $\underset{0\leq t\leq T}{\max} I_t$ and the total infection burden $\int_0^T I_t dt$, are recorded. For each value of the parameter being varied, the average percentage reductions in peak infection and total burden achieved by the optimal control is  then computed relative to the uncontrolled case.

\paragraph{Performance of the optimal protection with respect to the contagion rate.} Figure \ref{variationsbeta1} analyzes   the sensitivity of the control performance to changes in $\beta$, considering values in the range $\beta \in [0.01, 0.1]$, while keeping the attack intensity parameter $\nu$ fixed at $0.05$. It reports the average reductions  as functions of the contagion parameter $\beta$.
The reduction in the average number of infected individuals over the period $[0,T]$ increases monotonically with $\beta$, from $33.025\%$ to $36.11\%$. This metric accounts for the entire epidemic trajectory and therefore captures the cumulative effect of the protection strategy. As contagion intensifies, the sustained action of the control yields larger relative gains in terms of total infections avoided, resulting in a higher percentage reduction for larger values of $\beta$.
Concerning the peak  of infection, 
the reduction  varies from  $37.55\%$ to $38.11\%$  with a slight increase for small values of $\beta$ and then a slight decrease as $\beta$ becomes larger. This does not reflect a loss of control efficiency, but rather the intrinsic difficulty of acting on the epidemic peak when transmission is faster. For low values of $\beta$, the infection spreads more slowly, giving the control policy enough time to react and significantly reduce the peak level. When $\beta$ increases, the epidemic spreads more rapidly, and due to latency effects in the impact of the protection, the peak occurs earlier and at a relatively higher level before the protection fully takes effect. Nevertheless, the peak reduction remains within a relatively stable range across all values of $\beta$, indicating that peak mitigation remains effective even in more contagious regimes.

\paragraph{Performance of the optimal protection with respect to the attack rate.} 
Figure \ref{variationsnu1}  reports the average reductions  as functions of  the intensity of the attack $\nu$ when $\nu\in [0.01, 0.1]$ while holding $\beta$ at its reference value ($\beta=0.04$). We observe higher sensitivity than for the contagion parameter $\beta$. Indeed, the reduction in the average number of infected individuals over the period $[0,T]$ is increasing concave  in  $\nu$, from  a very low value when the attack stream is low ($\nu= 0.01$) and reaching a plateau at $32\%$ for intense attack value ($\nu=0.07 $). The larger the attack level, the higher the percentage  reduction of total burden.
Concerning the infection peak, the reduction is also concave, exhibiting a steady increase from $5.34\%$ (for $\nu = 0.01$) to $47.46\%$ (for $\nu = 0.062$), followed by a slight decrease to $41.4\%$ (for $\nu = 0.1$). Similarly to the sensitivity with respect to $\beta$, this decrease may be explained by the fact that, under high attack intensity, the peak is reached earlier and at a relatively higher level, before the protection fully takes effect.

\begin{figure}[htbp]
\centering
\includegraphics[width=0.7\textwidth]{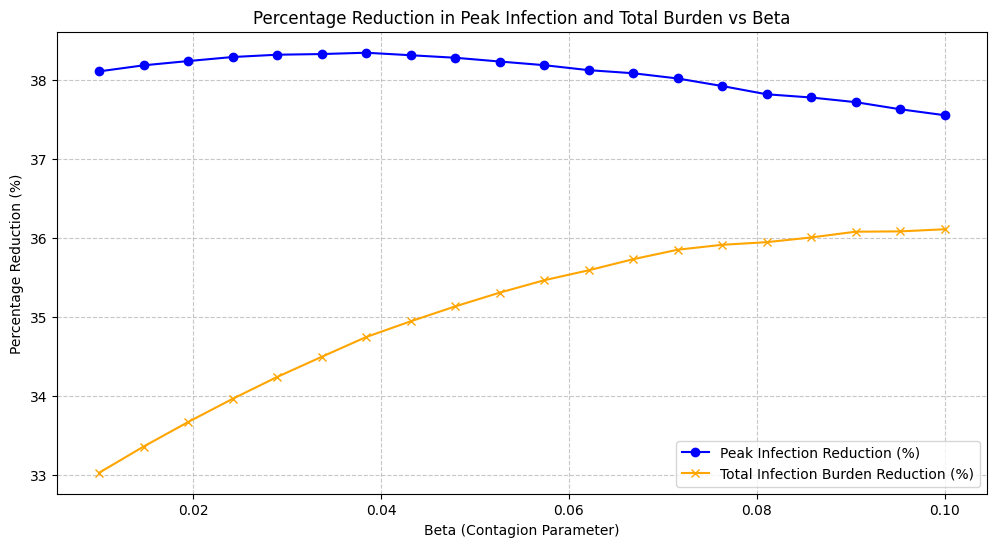}
\caption{Impact of contagion parameter $\beta$ on control performance. Scenario 1.}
\label{variationsbeta1}
\end{figure}

\begin{figure}[htbp]
\centering
\includegraphics[width=0.7\textwidth]{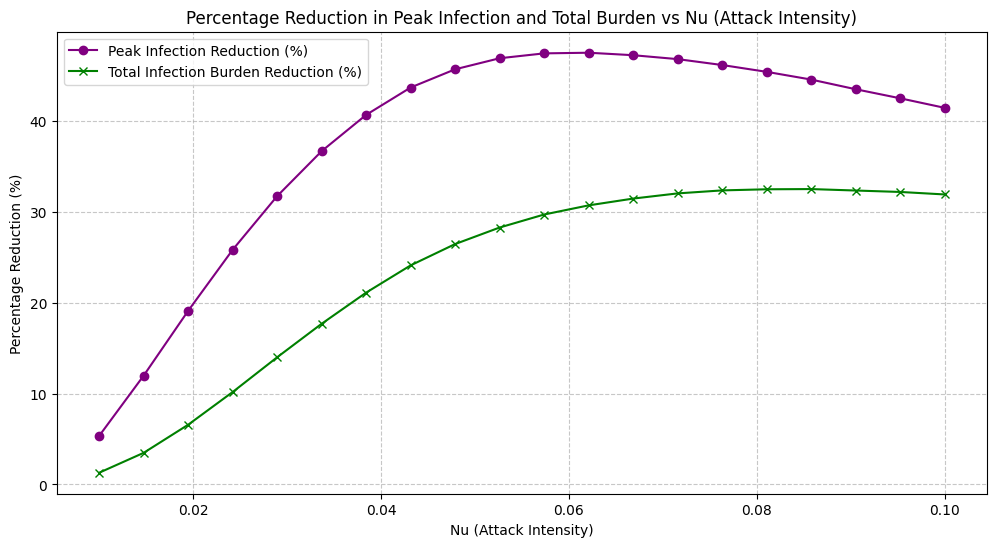}
\caption{Impact of attack parameter $\nu$ on control performance. Scenario 1.}
\label{variationsnu1}
\end{figure}

\subsubsection{Scenario 2: exogenous Poisson attacks}
We now illustrate Theorem \ref{thm:verification} when the attack engaging/disengaging times occur with a Poisson process of intensity $\lambda=0.1$, starting with an attack strategy at time $0$ in the state $a=1$ and no protection $p=0$. We choose the efficiency of the protection $\kappa= 0.02$, the marginal cost of protection is $c_V= 0.04 $ while the marginal cost of infected device is $c_I=0.01$. The switching costs to reinforces or relax the protection, that is going from $p=0$ to $p=1$ or conversely are $g_{01}=0.01 v(s,i,a,0)$ while $ g_{10}=0.001 v(s,i,a,1)$ for any $a\in \{0,1\}$. It means that there is a factor 10 of switching from no protection to protection strategy compared with the cost of the converse switch.\newline

We perform similar analysis as for Scenario 1. Figure \ref{figscenario2} represents one path of $S$, $I$ in Scenario 2 without protection and with optimal protection
for a fixed $\omega$. Similarly to Scenario 1,
we observe  that the cluster owner let the attack spreading a little bit before enhancing protection systems until time $\hat\tau_1=7.155=$ 7 days and 4 hours (first pink dotted vertical line). It is again explained by the switching cost to protect the system which is too high comparing to the cost of the infection at the beginning. Then, at time $\hat\tau_1=7.155$ the cluster owner enforces the protection. Randomly, at time $\tilde\tau_1=13.2=$ 13 days and 5 hours, the hacker disengages the attack (first blue dotted vertical line). After this time, the system stays in a state where there is no attack and the cluster owner is still protecting the system to contain even more efficiently the spread of the attack among the network. Then, the cluster owner disengages the protection system at time $\hat\tau_2=19.105=$ 19 days and 2 hours, before the next random attack at $\tilde\tau_2=22.35=$ 22 days and 8 hours (last blue dotted line). We see that despite the last attack at time  $\tilde\tau_2= 22.35$, the cluster owner does not reengage the protection system. It is explained by the successful management of the switching between protection and no protection strategy along time under the random attacks of the hacker, so that the attack and its spread is efficiently monitored. We observe that the final number of corrupted devices (red curve) is significantly lower (around 60\%) than without any protection strategy (brown curve).
\begin{figure}[htbp]
\centering
\includegraphics[width=0.7\textwidth]{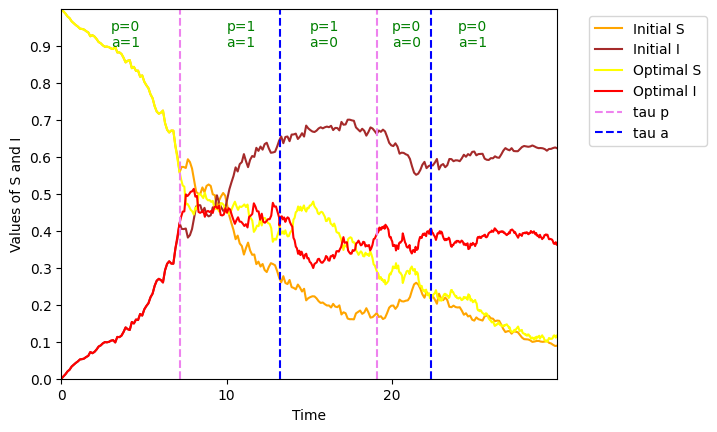}
    \caption{Optimal trajectory of $S$ and $I$ with protection and switching v.s. no protection strategy. Scenario 2.}
   \label{figscenario2}
\end{figure}

\newpage
 Figure \ref{variationsbeta2} (resp. Figure \ref{variationsnu2})  illustrates the infection reduction effectiveness by providing  the percentage of reduction of the two metrics (Peak of the infection and Total Burden) as functions of the contagion  parameter $\beta$ (resp. of the attack intensity parameter $\nu$). Similar patterns to those observed in Scenario 1 arise. Nevertheless, the functions are less regular and the protection is less efficient overall, which can be attributed to the randomness of the attack environment. For instance, with respect to the contagion parameter $\beta$, the percentage reduction range is $25-31\%$ for the infection peak, compared to about $38\%$ in Scenario 1, and  $12-24\%$, compared to $33-36\%$ for the total burden. Regarding the attack intensity parameter $\nu$, the protection in Scenario 2 is also less effective than in Scenario 1, but it increases in an almost linear way (rather than concave) and reaches approximately the same level of efficiency as Scenario 1 for large values of $\nu$.\newline

\begin{figure}[htbp]
\centering
\includegraphics[width=0.6\textwidth]{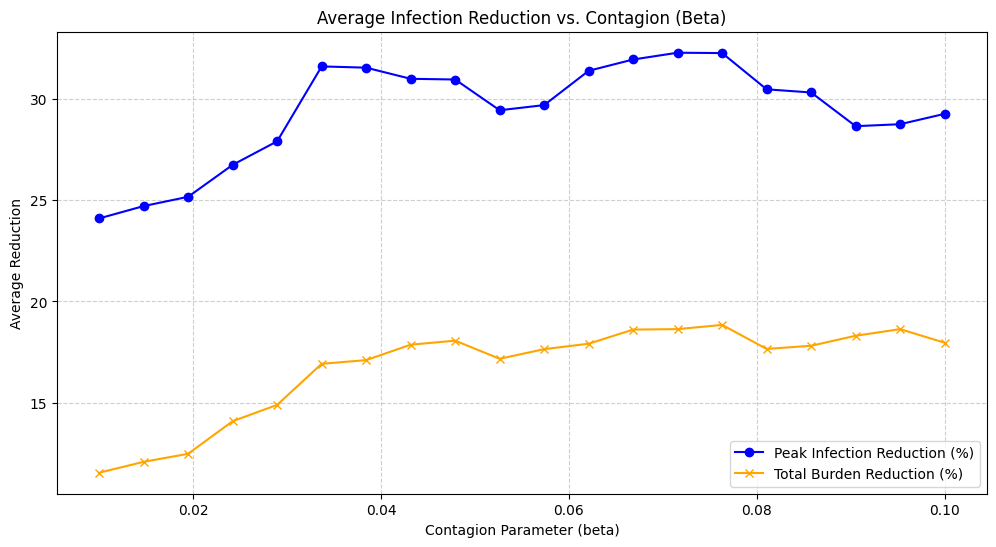}
\caption{Impact of contagion parameter $\beta$ on control performance. Scenario 2.}
\label{variationsbeta2}
\end{figure}

\begin{figure}[htbp]
\centering
\includegraphics[width=0.6\textwidth]{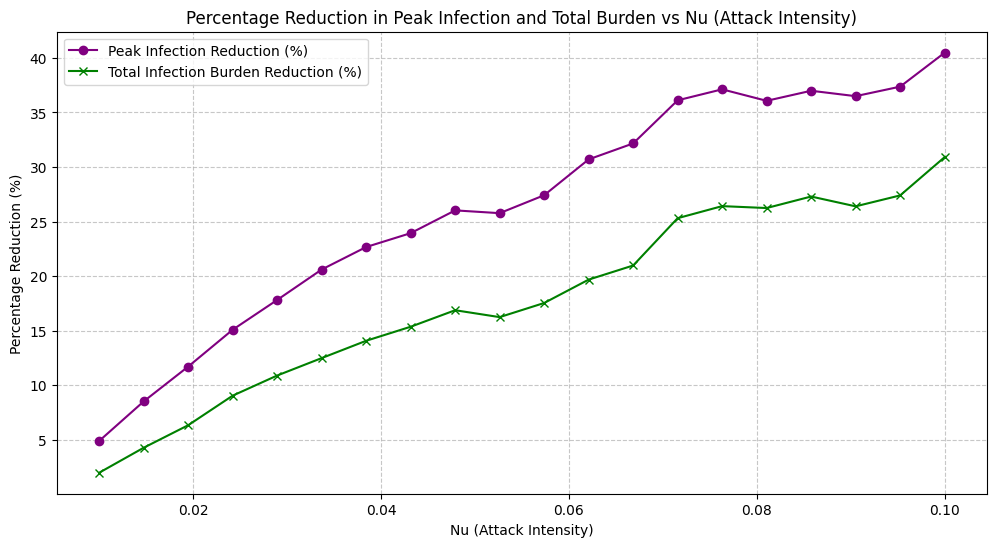}
\caption{Impact of attack parameter $\nu$ on control performance. Scenario 2.}
\label{variationsnu2}
\end{figure}

Overall, we observe that in the two scenarios, the cluster owner contains the corrupted devices efficiently at a final level around 30\% instead of 60\% (factor 0.5). We also observe a kind of robustness in the terminal value of $I$ when the system is optimally protected independently of the scenario. It suggests for future study to investigate more the behavior of the hacker, by finding a Nash equilibrium for the system hacker-cluster owner and the optimal attack/protection strategies chosen along the time period.

\section*{Conclusion}
Inspired by epidemiological models, the paper formulates a problem of cyber risk management through a two-dimensional stochastic switching control problem.   The optimal
protective measures are characterized as viscosity solutions to a system
of coupled variational inequalities, which are numerically approximated using  the Deep Galerkin algorithm. 
This work paves the way for the development of stochastic control methods in cyber risk management, a field that has so far received limited attention in the literature. 
In particular, this paper is a first milestone towards a more realistic game framework, in which the hacker is also strategic and whose optimal strategy is anticipated by the cluster owner.  Typically, the hacker aims to optimize the following trade-off: his value function is an increasing functions of the number of infected in the cluster, while  at the same time launching attacks is costly for him. The resolution of this  game framework is work in progress.

  \small
  \bibliographystyle{acm}

\section{Statements and Declarations}
 This work benefit from the support of the France-Berkeley Fund 2023,  of the ANR project DREAMeS (ANR-21-CE46-0002) and ANR ReLISCoP (ANR-21-CE40-0001). The authors have no relevant financial or non-financial interests to disclose. All authors listed contributed equally to this work. 
 \end{document}